\documentclass[hha]{article}
\usepackage{amsmath}
\usepackage{amssymb}
\usepackage{amsthm}
\usepackage[mathscr]{eucal}
\usepackage[all,knot,cmtip]{xy}
\xyoption{arc}

\numberwithin{equation}{section}
\newtheorem{definition}{Definition}[section]

\newtheorem{lem}[definition]{Lemma}

\newtheorem{rem}[definition]{Remark}

\newtheorem{thm}[definition]{Theorem}

\newtheorem{cor}[definition]{Corollary}
\newtheorem{ex}[definition]{Example}

\begin{document}
\title{On complicial homotopy monoids}
\author{Ryo Horiuchi}
\date{}

\maketitle

\section{Introduction}
For a Kan complex with a vertex and $n\geq 1$, we have the notion of its $n$-th homotopy group. This notion has been playing a central role in geometry. In this paper, for a weak complicial set in the sense of \cite{V1} with a vertex and $n\geq 1$, we construct a monoid, which we call its $n$-th homotopy monoid, and show that it is a natural generalization of homotopy group.

Our main theorem is the following:

\begin{thm}For a weak complicial set $X$, its vertex $x\in X$ and $n\geq 1$, we derive a monoid $\tau_n(X, x)$, which is not a group in general.

For a Kan complex $K$ and its vertex $v\in K$, we let $\operatorname{th_0}(K)$ denote the corresponding weak complicial set. Then the monoid $\tau_n(\operatorname{th_0}(K), v)$ coincides with the simplicial homotopy group $\pi_n(K, v)$.
\end{thm}

Stable homotopy theorists have been regarding the sphere spectrum $\mathbb{S}$ as an algebra deeper than the initial ring $\mathbb{Z}$ of integers. 
With this attitude, they have established the theory of higher algebras, which has remarkably been successful in these decades. 

For example, the most basic homology theory of higher algebras, topological Hochschild homology defined in \cite{B}, is known to be useful to compute other invariants such as algebraic $K$-theory and has contributed, in particular, to number theory and to arithmetic geometry (cf. \cite{BMS}, \cite{H}). 
One of the features of topological Hochschild homology is that it is related to Witt vectors. 
More precisely, for a commutative ring $R$, the topological Hochschild homology spectrum $\operatorname{THH}(R)$ and its cyclotomic structure give the ring $\mathbb{W}(R)$ of (big) Witt vectors and its structure morphisms (see \cite{HM} for this result, see \cite{NS1} and \cite{NS2} for cyclotomic structure in terms of $(\infty, 1)$-category theory). 
This is an important fact when using homology theories of higher algebras for other branches of mathematics. 
Another fact we emphasize here is that the theory of Witt vectors for commutative rings has naturally been generalized to that for commutative semirings in terms of universal algebras in \cite{Borger}. 
However, $\operatorname{THH}(R)$ for semirings $R$ do not make sense.

It is well known that the category of $\Gamma$-spaces can admit a model structure of connective spectra, where connective spectra are spectra whose negative stable homotopy groups vanish (see for example \cite{BF} or \cite{S} for more detail).
As is shown in \cite[\S 3]{CC}, the category of commutative monoids can be embedded in that of $\Gamma$-sets (i.e. discrete $\Gamma$-spaces) and this embedding induces the embedding of the category of semirings into that of monoid objects in $\Gamma$-sets. 
This fact implies that we can regard the initial semiring $\mathbb{N}$ as an algebra over $\mathbb{S}$, since $\mathbb{S}$ actually lives in the category of $\Gamma$-sets as the unit. 
This attitude plays a role in op. cit. However again, we are unable to analyze semirings in the homotopy theory of (connective) spectra, since everything is characterized by homotopy groups in it and semirings do not have their own underlying abelian groups in general.

Thus we may wonder if there can be a possibly new homotopy theory to study semirings as algebras over $\mathbb{S}$ or to extend the relation between topological Hochschild homology and Witt vectors for commutative rings to that for commutative semirings. 
Then we may have to ask if we can have the notion of Eilenberg-MacLane spectra of commutative monoids or semirings for such homotopy theory, and immediately after asking so we may again ask whether it is possible to establish a possibly new notion of space whose natural invariants take the values in monoids instead of groups in general or not.
 
In this paper, we do not answer such questions unfortunately, but using the theory of stratified simplicial sets we construct the notion of {\it homotopy monoid} which is a natural generalization of homotopy group. 
We hope this short article takes the first step to answer the naive questions above.

A stratified simplicial set is a pair of a simplicial set and a subset of its simplices with certain conditions. In \cite{V1} Verity constructed a model structure on the category of stratified simplicial sets, in which fibrant and cofibrant objects are precisely weak complicial sets, and showed that any Kan complex can be viewed as a weak complicial set. That is to say, the model structure is a generalization of that for $\infty$-groupoids. Therefore it should be natural to try to generalize the notion of simplicial homotopy groups of Kan complexes to weak complicial sets. 

The intuition used in this short article is that, as in a Kan complex every simplex is ``invertible", in a weak complicial set every {\it thin} simplex is ``invertible". We see that under this intuition we can simply apply to weak complicial sets the analogous construction of simplicial homotopy group.

%%%%%
%二章
%%%%%

\section{Preliminaries}
Assuming the reader is familiar with simplicial sets, we recall some notations about weak complicial sets from \cite{V1}.

\begin{definition}[\cite{V1}]A pair $(X, tX)$ is a stratified simplicial set\footnote{This may be called marked simplicial set or simplicial set with marking in some literature. However the word marked simplicial set seems to be used for a different notion as well. We follow the nomenclature of \cite{V1} to avoid confusion.} if
\begin{itemize}
  \item $X$ is a simplicial set,
  \item $tX$ is a set of simplices in $X$ such that $dX\subset tX$ and $X_0\cap tX=\emptyset$,
\end{itemize}
where $dX$ denotes the set of degenerate simplices in $X$.

Let $(X, tX)$ and $(Y, tY)$ be stratified simplicial sets. A stratified map $f:(X, tX)\to (Y, tY)$ is a simplicial map $f: X\to Y$ such that $f(x)\in tY$ for all $x\in tX$.
\end{definition}

We call $X$ the underlying simplicial set of $(X, tX)$ and elements in $tX$ its thin simplices. However, for the simplicity, we often write $X$ for a stratified simplicial set $(X, tX)$ omitting $tX$. Also we denote the category of stratified simplicial sets and stratified maps by ${\rm\underline{Strat}}$. 

\begin{ex}{\rm Every simplicial set $X$ defines stratified simplicial sets $(X, dX)$ and $(X, \bigcup_{n\geq 1}X_n)$. Each assignment gives rise to a functor ${\rm\underline{Simp}}\to{\rm\underline{Strat}}$, which is left (resp. right) adjoint to the forgetful functor, where ${\rm\underline{Simp}}$ denotes the category of simplicial sets and simplicial maps.}
\end{ex}

As the standard simplicial sets and their horns play a role, in particular they give the definition of quasi-category, in the theory of simplicial sets, we need the following specific stratified simplicial sets. For stratified simplicial sets $(X, tX)$ and $(Y, tY)$, we say that $(X, tX)$ is a regular stratified simplicial subset of $(Y, tY)$ if $X\subset Y$ as simplicial sets and $tX=X\cap tY$ (cf. \cite{V1}).

%基本的な印付単体的集合のリスト

\begin{definition}[\cite{V1}]Let $n$ be a natural number and $k\in[n]$.
\begin{itemize}

  \item The standard thin $n$-simplex $\Delta[n]_t$ is the stratified simplicial set whose underlying simplicial set is the standard simplicial set $\Delta[n]$ and  
  \[
  t\Delta[n]_t = \begin{cases}
    d\Delta[n]\cup\{\operatorname{Id}_{[n]}\} & (n\neq 0) \\
    d\Delta[n] & (n=0)
  \end{cases}
\]
  
  \item The $k$-complicial $n$-simplex $\Delta^k[n]$ is the stratified simplicial set whose underlying simplicial set is the standard simplicial set $\Delta[n]$ and
  \[t\Delta^k[n]=d\Delta[n]\cup\{\alpha\in\Delta[n]| \{k-1, k, k+1\}\cap[n]\subset\operatorname{Im}(\alpha)\}\] 
  
  \item The $n-1$-dimensional $k$-complicial horn\footnote{If we write $\Lambda^k[n]_{simp}$ for the simplical horn, this stratified simplicial set $\Lambda^k[n]$ is different from both of $(\Lambda^k[n]_{simp}, d\Lambda^k[n]_{simp})$ and $(\Lambda^k[n]_{simp}, \bigcup_{n\geq1}\Lambda^k[n]_{simp}$) in general. The underlying simplicial set of $\Lambda^k[n]$ is $\Lambda^k[n]_{simp}$.} $\Lambda^k[n]$ is the regular stratified simplicial subset of $\Delta^k[n]$ generated by the set of faces $\{\delta_i|i\in[n]\setminus k\}$
  
  \item $\Delta^k[n]''$ (resp. $\Lambda^k[n]'$) is the stratified simplicial set whose underlying simplicial set is the same as that of $\Delta^k[n]$ (resp. $\Lambda^k[n]$) and its thin simplices are $t\Delta^k[n]$ (resp. $t\Lambda^k[n]$) with all its $n-1$-simplices
  
  \item $\Delta^k[n]':=\Delta^k[n]\cup\Lambda^k[n]'$.
  \end{itemize}
\end{definition}

These stratified simplicial sets define the notion of weak complicial set, which is the subject of this section.

\begin{definition}[\cite{V1}]A stratified simplicial set is called a weak complicial set\footnote{This is also called (non-saturated) complicial set in some literature. But again we follow the nomenclature in \cite{V1}.} if it has the right lifting property with respect to the following morphisms:
\begin{itemize}
  \item $\Lambda^k[n]\hookrightarrow\Delta^k[n]$ for $n\geq 1$ and $k\in[n]$,
  
  \item $\Delta^k[n]'\hookrightarrow\Delta^k[n]''$ for $n\geq 2$ and $k\in[n]$.
\end{itemize}
\end{definition}

In \cite{V1}, it is shown that every quasi-category can be vied as a weak complicial set. Moreover, in \cite{V2}, it is shown that every strict $\omega$-category can be viewed as a weak complicial set (via Street's $\omega$-nerve functor). Therefore weak complicial set is a common generalization of $(\infty, 1)$-category and strict $\omega$-category.

In particular weak complicial set is a generalization of $\infty$-groupoid that is homotopy theoretically equivalent to topological space, so we may take weak complicial sets as spaces in which (higher) cells are not necessarily invertible, while in any $\infty$-groupoid every $n$-cell ($n\geq1$) must be invertible.

In addition, Verity constructed in \cite{V1} a model structure on ${\rm\underline{Strat}}$, in which weak complicial sets are precisely the fibrant and cofibrant objects. Hence we already have a homotopy theory of weak complicial sets. Note that in op. cit. the weak equivalences of stratified simplicial sets is defined without using homotopy monoids we are constructing in the next section.

Before going to the next section, we recall the cartesian product of stratified simplicial sets.

\begin{definition}[\cite{V1}]Let $X$ and $Y$ be stratified simplicial sets. Then the cartesian product $X\circledast Y$ of them is a stratified simplicial set whose underlying simplicial set is $X\times Y$ and a simplex $(x, y)\in X\circledast Y$ is thin $:\overset{def}{\Leftrightarrow}$ $x\in tX$ and $y\in tY$.
\end{definition}

%%%%%%%%%%%%%%%%
%三章
%%%%%%%%%%%%%%%%

\section{Complicial construction of homotopy monoids}

We use the notations recalled in the previous section to construct homotopy monoids referring to famous textbooks such as \cite{GJ} and \cite{M}.

\begin{definition}[\cite{V1}]Let $f, g:A\to X$ be stratified maps of stratified simplicial sets. We write $f\sim g$ if there exits a map $H:\Delta[1]_t\to X$ such that 
   \[
   \xymatrix{
 A\circledast\Delta[0]\ar[d]_{1_A\times d^1}\ar[r]^{\ \ \cong}&A\ar[rd]^f&\\
 A\circledast\Delta[1]_t\ar[rr]^H&&X\\
 A\circledast\Delta[0]\ar[u]^{1_A\times d^0}\ar[r]_{\ \ \cong}&A \ar[ru]_g&     
}
\]
commutes.
\end{definition}

We may call this $H$ a (simple) homotopy from $f$ to $g$. Since our aim is to construct homotopy monoids, we may need the notion of relative homotopy as well.

\begin{definition}Let $f, g:A\to X$ be stratified maps of stratified simplicial sets and $B\to A$ an inclusion of stratified simplicial sets. Assume that $f|_B=g|_B$. We write $f\sim_B g$ if $f\sim g$ with $H: A\circledast\Delta[1]_t\to X$ and 
   \[
   \xymatrix{
    A\circledast\Delta[1]_t\ar[r]^{\ \ \ H}&X\\
 B\circledast\Delta[1]_t\ar@{^{(}->}[u] \ar[r]_{\ \ \ proj}  &B\ar[u]_{f|_B=g|_B}
}
\]
commutes.
\end{definition}

\begin{lem}The relation $\sim$ is an equivalence relation for vertices in a weak complicial set.
\end{lem}

\begin{proof}Let $X$ be a weak complicial set. For any vertex $x$ in $X$ we can take the constant 1-simplex at $x$ that is thin, then $\sim$ is reflexive.

Assume $x\sim y$ and $y\sim z$ with $x, y, z$ are vertices of $X$. Then we have thin 1-simplices $H$ from $x$ to $y$ and $H'$ from $y$ to $z$. These give rise to a map $\Lambda^1[2]\to X$ which lifts to $\Delta^1[2]\to X$ since $X$ is a weak complicial set. Thus we obtain a 1-simplex from $x$ to $z$. As $H$ and $H'$ are thin, this is actually a map from $\Delta^1[2]'$. Hence eventually we have a map $\Delta^1[2]''\to X$ since $X$ is a weak complicial set. This map gives a thin 1-simplex from $x$ to $z$ to show that $\sim$ is transitive.

Let $x\sim y$ with a thin 1-simplex $H$. Then we have a map $\Lambda^0[2]\to X$ which maps 2-face to $H$ and 1-face to the constant of $x$. Since $X$ is a weak complicial set, this map defines a map $\Delta^0[2]\to X$. Since both of the homotopy $H$ and the constant 1-simplex are thin, this map indeed is a map $\Delta^0[2]'\to X$ and again $X$ is a weak complicial set, this lifts to a map $\Delta^0[2]''\to X$ to give a thin 1-simplex from $y$ to $x$.
\end{proof}

This generalizes to higher simplexes due to the cartesian closedness of weak complicial sets, which is proven in \cite{V1}.

\begin{lem}\footnote{The first part of this result is mentioned in \cite{R}}For stratified maps $A\to X$ with $X$ a weak complicial set, the relation $\sim$ is an equivalence relation. Moreover, if $B\to A$ is an inclusion of stratified simplicial sets, the relation $\sim_B$ is also an equivalence relation for stratified maps $A\to X$ which coincide each other on $B$.
\end{lem}
\begin{proof}By theorem 75 in \cite{V1}, the closure map $\operatorname{hom}(A, X)\to\operatorname{hom}(B, X)$ is a complicial fibration between weak complicial sets. Since vertexes in $\operatorname{hom}(A, X)$ correspond to maps $A\to X$, the lemma above proves this one.  
\end{proof}

Then we can define homotopy monoids as follows.

\begin{definition}Let $X$ be a weak complicial set, $x\in X$ a vertex and $n\geq 1$. Then we define the $n$-th homotopy monoid $\tau_n(X, x)$ to be the set of equivalence classes under $\sim_{\partial\Delta[n]}$ of $n$-simplexes $\alpha$ in $X$ such that
   \[
   \xymatrix{
    \Delta[n]\ar[r]^{\alpha}&X\\
 \partial\Delta[n]\ar@{^{(}->}[u] \ar[r]  &\Delta[0]\ar[u]_{x}
}
\]
commutes.
\end{definition}

We are going to construct a monoid structure on this set. Consider two $n$-simplices $\alpha$ and $\beta$ in $X$ such that $\alpha|_{\partial\Delta[n]}=\beta|_{\partial\Delta[n]}$ is the constant at $x$. We can construct a stratified map $\Lambda^n[n+1]\to X$ such that $n-1$-face maps to $\alpha$, $n+1$-face maps to $\beta$ and other faces map to the constant $x$. Since $X$ is a weak complicial set, this lifts to a map $\theta:\Delta^n[n+1]\to X$. In particular we obtain an $n$-simplex $d_n(\theta)$ with $d_n(\theta)|_{\partial\Delta[n]}$ is the constant at $x$.

\begin{rem}
{\rm Note that the non-degenerate $n+1$-simplex $\operatorname{Id}_{[n+1]}$ in $\Delta^n[n+1]$ is thin. Thus the $n+1$-simplex ``between $\alpha$ and $\beta$" is thin. For example, when $n=1$, we have the following picture:
   \[
   \xymatrix{
   & x&\\
 x\ar_{\beta}[rr]\ar^{d_n(\theta)}[ru]&  &x,\ar_{\alpha}[lu]
}
\]
where the $n+1$-simplex surrounded by $n$-simplexes $\alpha$, $\beta$ and $d_n(\theta)$ (and constants) is thin.}
\end{rem}

%この補題の証明に訂正を入れた

\begin{lem}With the notation above, the class $[d_n(\theta)]$ is independent of the choices of representatives of $[\alpha]$ and $[\beta]$ and that of $\theta$.
\end{lem}

\begin{proof}Suppose $\alpha\sim_{\partial\Delta[n]}\alpha'$ with a homotopy $H$ and $\beta\sim_{\partial\Delta[n]}\beta'$ with a homotopy $H'$. Then, as we just saw, there are maps $\theta, \theta':\Delta^n[n+1]\to X$. 
Then we construct a map 
\[\Delta^n[n+1]\circledast\partial\Delta[1]\to X\]
which is $\theta$ (resp. $\theta'$) when restricted to $\Delta^n[n+1]\circledast 0$ (resp. $\Delta^n[n+1]\circledast 1$), where $\partial\Delta[1]$ denotes the stratified simplicial set $(\partial\Delta[1], d\partial\Delta[1])$.
  
Also we construct a map 
\[\Lambda^n[n+1]\circledast\Delta[1]\to X\]
using $\alpha$, $\alpha'$, $H$, $\beta$, $\beta'$ and $H'$,
where $\Delta[1]$ denotes the stratified simplicial set $(\Delta[1], d\Delta[1])$.
More precisely, when restricted to $\Lambda^n[n+1]\circledast 0$ (resp. $\Lambda^n[n+1]\circledast 1$) this map is the one that $n-1$-face maps to $\alpha$ (resp. $\alpha'$), $n+1$-face maps to $\beta$ (resp. $\beta'$) and other faces map to the constant $x$.

These maps give rise to a map 
\[(\Delta^n[n+1]\circledast\partial\Delta[1])\cup(\Lambda^n[n+1]\circledast\Delta[1])\to X.\]

By lemma 72 in \cite{V1}, 
\[(\Delta^n[n+1]\circledast\partial\Delta[1])\cup(\Lambda^n[n+1]\circledast\Delta[1])\to \Delta^n[n+1]\circledast\Delta[1]\]
is a left anodyne extension.
Furthermore, looking at 1-simplices of $\Delta^n[n+1]\circledast\Delta[1]$ and those of $\Delta^n[n+1]\circledast\Delta[1]_t$,  we see that 
\[\Delta^n[n+1]\circledast\Delta[1]=\Delta^n[n+1]\circledast\Delta[1]_t.\]
More precisely, the underlying simplicial sets of them are the same, which is $\Delta[n+1]\times\Delta[1]$. By definition $(\alpha, \beta)\in\Delta[n+1]\times\Delta[1]$ is thin in $\Delta^n[n+1]\circledast\Delta[1]_t$ (resp. in $\Delta^n[n+1]\circledast\Delta[1]$) if and only if $\alpha$ is thin in $\Delta^n[n+1]$ and $\beta$ is thin in $\Delta[1]_t$ (resp. in $\Delta[1]$). Again by definition $t\Delta[1]_t\setminus t\Delta[1]=\{\operatorname{Id}_{[1]}\}$ and there is no thin 1-simplex in $\Delta^n[n+1]$ since $n\geq1$.

Therefore, since $X$ is a weak complicial set, we obtain a map 
\[\Delta^n[n+1]\circledast\Delta[1]_t\to X\]
 to define a homotopy from $d_n(\theta)$ to $d_n(\theta')$.
\end{proof}

Thus we can define a multiplication on $\tau_n(X, x)$ by $[\alpha][\beta]=[d_n(\theta)]$.

\begin{thm}This multiplication gives rise to a monoid structure on $\tau_n(X, x)$.
\end{thm}

\begin{proof}First we show that this multiplication is associative. Let $\alpha$, $\beta$ and $\gamma$ represent elements in $\tau_n(X, x)$. As above, we obtain an $n+1$-simplex $\theta$ from $\alpha$ and $\beta$, an $n+1$-simplex $\psi$ from $d_n(\theta)$ and $\gamma$, and an $n+1$-simplex $\phi$ from $\beta$ and $\gamma$.

Referring to the remark above, we see that these data give rise to a map $\Lambda^n[n+2]\to X$ such that $n-1$-face maps to $\theta$, $n+1$-face to $\psi$, $n+2$-face to $\phi$ and other faces to $x$. Since $X$ is a weak complicial set, this lifts to a map $u:\Delta^n[n+2]\to X$. This shows that our multiplication is associative as follows:
\begin{eqnarray*}
   ([\alpha][\beta])[\gamma]&=& [d_n(\theta)][\gamma] \\
   &=& [d_n(\psi)]\\
   &=&[d_nd_n(u)]\\
   &=&[\alpha][d_n(\phi)]\\
   &=&[\alpha]([\beta][\gamma]),
\end{eqnarray*}
where we use the simplicial identity at the third $``="$ and use the definition of our multiplication of $[\alpha]$ and $[d_n(\phi)]$ at the fourth $``="$.

Note that the constant at $x$ defines the unit $e$, then we obtain a monoid structure on $\tau_n(X, x)$.
\end{proof}

\begin{rem}
{\rm This monoid structure on $\tau_n(X, x)$ is not necessarily a group structure. For example when $n=1$ consider the following picture:
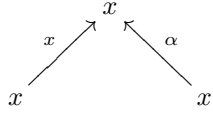
\[
   \xymatrix{
   & x&\\
 x\ar^{x}[ru]&  &x\ar_{\alpha}[lu]
}
\]
If $\alpha$ is thin, this picture will be given by a map $\Lambda^2[2]\to X$, then it will lift to a map $\Delta^2[2]\to X$:
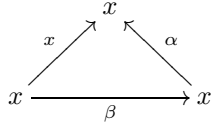
\[
   \xymatrix{
   & x&\\
 x\ar^{x}[ru]\ar_{\beta}[rr]&  &x\ar_{\alpha}[lu]
}
\]
This may show that $[\alpha][\beta]=e$. However, when $\alpha$ is not thin, we do not find its right inverse. The dual argument works for the left inverses and a similar argument works for higher $n$.

However again, as \cite[Example 16]{V1} shows that we can view Kan complexes as weak complicial sets. More precisely, for a Kan complex $K$, we obtain the stratified simplicial set $(K, \bigcup_{n\geq1}K_n)$, which is a weak complicial set by definition. We let $\operatorname{th_0}(K)$ denote the weak complicial set.  Note that, by definition, all $n$-simplices with $n\geq 1$ in $\operatorname{th_0}(K)$ are thin.}
\end{rem}

\begin{cor}For a Kan complex $K$, its vertex $v$ and $n\geq 1$, its homotopy group $\pi_n(K, v)$ and $\tau_n(\operatorname{th_0}(K), v)$ are the same as group.
\end{cor}

By the same observation we have the following as well.

\begin{cor}Let $m\geq 1$ be a natural number, $X$ be a weak complicial set whose all $k$-simplices are thin with $k\geq m$, and $x\in X$ a vertex. Then $\tau_n(X, x)$ is a group for $n\geq m$. 
\end{cor}

\begin{proof}Let $\beta$ be an $n$-simplex which represents an element of $\tau_n(X, x)$. By assumption it is thin. Thus it gives rises to a map $\Lambda^{n-1}[n+1]\to X$ which maps $n+1$-face to $\beta$ and others to the constants at $x$. Since $X$ is a weak complicial set, this map lifts to a map $\Delta^{n-1}[n+1]\to X$ which gives the left inverse of $[\beta]$. The dual arguments may give the right inverse of $\beta$.
\end{proof}

Example 57 of \cite{V1} shows that every quasi-category can be viewed as a weak complicial set via the functor $(\mathchar`-)^e$ from the category of quasi-categoies to that of weak complicial sets. More precisely, for a quasi-category $C$, we obtain the stratified simplicial set $(C, dC\cup\bigcup_{n\geq2}C_n)$ and make its specific 1-simplices thin to obtain a stratified simplicial set $C^e$. Theorem 56 in \cite{V1} shows that it is a weak complicial set. Note that, by construction, any $n$-simplex in $C^e$ with $n\geq 2$ is thin.

\begin{cor}Let $C$ be a quasi-category and $c\in C$ a vertex. Then the homotopy monoid $\tau_n(C^e, c)$ has the group structure defined above when $n\geq 2$.
\end{cor}

As is shown in \cite[Chapter 10]{V2}, the $\omega$-nerve functor has a left adjoint functor $\operatorname{F}_{\omega}:{\rm\underline{Strat}}\to{\rm \omega\mathchar`-\underline{Cat}}$, where ${\rm \omega\mathchar`-\underline{Cat}}$ denotes the category of strict $\omega$-categories and strict $\omega$-functors. For a weak complicial set $X$ with a vertex $x$, it may be natural to compare our monoids $\tau_*(X, x)$ and (higher) endomorphism monoids on $x$ in $\operatorname{F}_{\omega}(X)$.

%Note that in the first chapter of \cite{J} Joyal has defined the fundamental category of a quasi-category and hence the fundamental monoid of a pointed quasi-category. More precisely, for a quasi-category $C$ and a vertex $c\in C$, a category $\tau_1(C)$ is defined by the left adjoint of the nerve functor and called the fundamental category of $C$. Then we may obtain the endomorphism monoid $\operatorname{End}_{\tau_1(C)}(c)$. It may be reasonable to compare $\tau_1(C^e, c)$ and $\operatorname{End}_{\tau_1(C)}(c)$.

Note that so far we do not know whether higher homotopy monoids are commutative in general or not, although as a classical result we know that higher homotopy monoids for Kan complexes, which are homotopy groups, are commutative. It is shown in \cite{OR}, \cite{ORV} and \cite{R} that {\it $n$-trivial saturated weak complicial sets} is a model of $(\infty, n)$-categories for $n\in\mathbb{N}$. Hence it would be worth studying homotopy monoids of saturated ones rather than non-saturated weak complicial sets.

Finally, for a weak complicial set $X$, we define $\tau_0(X)$ to be the quotient set of $X_0$ divided by the equivalence relation $\sim$. Then, by definition, for a Kan complex $K$, $\tau_0(\operatorname{th_0}(K))=\pi_0(K)$, where $\pi_0(K)$ denotes the set of connected components in $K$. So we may call $\tau_0(X)$ the set of invertibly connected components in $X$.

\section{Acknowledgments}
This work is a part of a project suggested by Lars Hesselholt when I was a student supervised by him.  I appreciate him guiding myself to this project.

\end{document}